\documentclass[onefignum,onetabnum]{siamart190516}

\usepackage{amsfonts}
\usepackage{graphicx}
\usepackage{epstopdf}
\usepackage{algorithmic}
\usepackage{enumitem}
\ifpdf
  \DeclareGraphicsExtensions{.eps,.pdf,.png,.jpg}
\else
  \DeclareGraphicsExtensions{.eps}
\fi

\theoremstyle{plain}
\newtheorem{example}[theorem]{Example}
\providecommand{\abs}[1]{\lvert#1\rvert}
\providecommand{\norm}[1]{\lVert#1\rVert}

\newsiamremark{remark}{Remark}

\headers{Matrix H\"older's inequality and optimal transport}{K. J. Ciosmak}

\title{Matrix H\"older's inequality and divergence formulation of optimal transport of vector measures\thanks{The paper is based on the doctoral thesis \cite{Ciosmak4} of the author.
\funding{The financial support of St John's College in Oxford, the Clarendon Fund and EPSRC is gratefully acknowledged.  Part of this research was completed in Fall 2017 while the author was member of the Geometric Functional Analysis and Application program at MSRI, supported by the National Science Foundation under Grant No. 1440140. This research was also partly supported by the ERC Starting Grant 802689 CURVATURE.}}
}
\author{Krzysztof J. Ciosmak\thanks{University of Oxford, Mathematical Institute 
  (\email{ciosmak@maths.ox.ac.uk}) and 
University of Oxford, St John's College (\email{krzysztof.ciosmak@sjc.ox.ac.uk}).}
}
\usepackage{amsopn}

\ifpdf
\hypersetup{
  pdftitle={Matrix H\"older's inequality and divergence formulation of optimal transport of vector measures},
  pdfauthor={K.J. Ciosmak}
}
\fi

\begin{document}

\maketitle

\begin{abstract}
  We characterise equality cases in matrix H\"older's inequality and develop a divergence formulation of optimal transport of vector measures. As an application, we reprove the representation formula for measures in the polar cone to monotone maps. We generalise the last result to a wide class of polar cones, including polar cones to tangent cones to the unit ball in the space of differentiable functions and in the Sobolev spaces.
\end{abstract}



\section{Introduction}

\subsection{Matrix H\"older's inequality}

We shall present matrix H\"older's inequality with a proof taken from \cite{Baumgartner}. We characterise the equality cases for this inequality. According to the knowledge of the author, this characterisation is new.

Matrix H\"older's inequality is an extension of the classical H\"older's inequality to matrices. It states that for given two real matrices $A,B$ of size $m\times n$, there is
\begin{equation*}
\mathrm{tr}(A^*B)\leq \norm{A}_p \norm{B}_q,
\end{equation*}
where $p,q\in [1,\infty]$ are such that $\frac1p+\frac1q=1$. Here, for $p\in [1,\infty]$, $\norm{A}_p$ denotes the Schatten $p$-norm, defined as
\begin{equation*}
\norm{A}_p=\Big(\sum_{i=1}^{n\wedge m}a_i^p\Big)^{\frac1p},
\end{equation*}
where $(a_i)_{i=1}^{n\wedge m}$ denote the sequence of singular values of $A$. Equivalently,
\begin{equation*}
\norm{A}_p=(\mathrm{tr}|A|^p)^{\frac1p},
\end{equation*}
where $|A|=(A^*A)^{\frac12}$.

\subsection{Optimal transport of vector measures}

We provide a divergence formulation of optimal transport of vector measures. Given a vector-valued measure $\mu\in\mathcal{M}(\mathbb{R}^n,\mathbb{R}^m)$, such that $\mu(\mathbb{R}^n)=0$, we consider a variational problem of finding the infimum
\begin{equation}\label{eqn:measure}
\inf\Big\{\norm{M}_{\mathcal{M}}\mid M\in\mathcal{M}(\mathbb{R}^n,\mathbb{R}^{m\times n}), -div M=\mu\Big\}.
\end{equation}
We show that this value is equal to 
\begin{equation}\label{eqn:lipschitz}
\sup\Big\{\int \langle f,d\mu\rangle\mid f\colon \mathbb{R}^n\to\mathbb{R}^m, \norm{f}_{\mathcal{C}^1}\leq 1\Big\}.
\end{equation}
Here $M$ is a matrix-valued measure and $\norm{M}_{\mathcal{M}}$ stands for the total variation norm of $M$ with respect to the Schatten $1$-norm.  This constitutes an extension of the divergence formulation of optimal transport to the setting of vector measures; see \cite[1.2.3]{Villani2} for the case of signed measures. 

Problem of minimising the total variation of a signed measure with prescribed divergence is equivalent to the Monge--Kantorovich problem \cite{Kantorovich, Monge}.
Such divergence formulation of the optimal transport problem, or the flow-minimisation problem, has been first investigated by Beckmann \cite{Beckmann2}. We refer the reader to the book of Santambrogio \cite[Chapter 4]{Santambrogio} for a thorough discussion of the problem.

We develop an analogous theory for absolutely continuous vector measures. In there (\ref{eqn:lipschitz}) is replaced by an optimisation problem concerning the unit ball of the Sobolev space, whereas (\ref{eqn:measure}) is replaced by a minimisation over absolutely continuous matrix-valued measures. We develop a duality theory and provide, using the characterisation of equality cases in matrix H\"older's inequality, necessary and sufficient conditions for a matrix-valued absolutely continuous measure and a map in the Sobolev space to be extremisers in the corresponding problems.

Let us mention another formulation of optimal transport of vector measures; see \cite{Ciosmak2}. The advantage of the current formulation to the one in \cite{Ciosmak2} is that the primal problem's value (\ref{eqn:measure}) is always attained.

For yet another formulation of optimal transport problem of vector measures, developed by Chen, Georgiou, Tannenbaum, Haber, Yamamoto, Ryu, Li and Osher, see \cite{Chen1,Chen3,Chen2}.

\subsection{Michell trusses}

Let us note that our results are related to the works  in elasticity of Bouchitt\'e, Gangbo and Seppecher \cite{Gangbo1}, Gangbo \cite{Gangbo2}, see also \cite{Michell} for the original paper of Michell. 

In there the Michell trusses are investigated.
Suppose that we are given a measure, which represents the distribution of forces. We would like to design an optimal structure of finitely many bars, i.e. a truss, that can support such load. This condition may be restated by saying that the divergence of the stress of the truss
has to be equal to the distribution of forces.
The optimality condition is that the total stress is minimal.
 This problem is known to have no solution in general, as an sequence of  trusses approximating the optimal value of total stress, may lead to a diffuse structure. Appropriate relaxation leads to consideration of an optimisation problem of the form similar to (\ref{eqn:measure}).
 Here the assumption of the distribution of forces is that the total force and the total torque vanishes, which differs from the setting of our paper.
 
 \subsection{Representation formulae for polar cones}
 
As an application of our developments, we provide several representation theorems of polar cones. A polar cone to a subset $A$ of a Banach space is defined as a set of all linear, continuous functionals on the Banach space which are non-positive on the set $A$. A tangent cone to a convex set $A$ at a point $x\in A$, is defined as the closure of the set of all points of the form $\lambda(a-x)$ for $\lambda\geq 0$ and $a\in A$. Therefore, a polar cone to a tangent cone to $A$ at $x\in A$ is the set of linear, continuous functionals that attain their maximum on $A$ at the point $x$.

As an outcome of our study we reprove the result of Cavalletti and Westdickenberg from \cite[Theorem 2.1]{Cavalletti4}.  This is to say, we provide a novel proof of the representation formula for vector measures that belong to the polar cone to the set of monotone maps. Let us recall that a map $u\colon\mathbb{R}^n\to\mathbb{R}^n$ is monotone whenever for all $x,y\in\mathbb{R}^n$ there is
\begin{equation*}
\langle u(x)-u(y),x-y\rangle\geq 0.
\end{equation*}
In essence, both proofs rely on the Hahn--Banach theorem. Yet, in our proof, we apply it directly, to obtain the duality results -- Theorem \ref{thm:duality3} and Theorem \ref{thm:dualityLebesgue}. In \cite{Cavalletti4}, Riesz' extension theorem is employed, which provides extensions of positive functionals. The advantage of our proof is that the representation of elements of polar cone needs not to be known a priori, unlike in the proof of \cite{Cavalletti4}. This allows for application of our methods to computation of representation formulae for elements of other polar cones to tangent cones of the unit balls. In particular, we provide such formulae for the case of the space of differentiable maps and for the case of the Sobolev spaces for exponents in the reflexive range.

In \cite{Cavalletti4} vector-valued measures $\mu\in\mathcal{M}(\mathbb{R}^n,\mathbb{R}^n)$ and matrix-valued measures $H\in\mathcal{M}(\mathbb{R}^n,\mathbb{R}^{n\times n})$ such that for all monotone $u\in\mathcal{C}^1(\mathbb{R}^n,\mathbb{R}^n)$ there is
\begin{equation}\label{eqn:polari}
-\int_{\mathbb{R}^n}\langle u,d\mu\rangle-\int_{\mathbb{R}^n}\langle e(u),dH\rangle\geq 0,
\end{equation}
are considered. Here $e(u)=\frac12( Du+Du^*)$ is called the deformation tensor of $u$. It is shown that then exists a measure $M$, with values in symmetric and positive semidefinite matrices, such that for all smooth and compactly supported $u$ there is
\begin{equation}\label{eqn:def}
-\int_{\mathbb{R}^n}\langle u,d\mu\rangle-\int_{\mathbb{R}^n}\langle e(u),dH\rangle=\int_{\mathbb{R}^n}\langle e(u),dM\rangle
\end{equation}
and
\begin{equation*}
\int_{\mathbb{R}^n}\mathrm{tr}(dM)\leq -\int_{\mathbb{R}^n}\langle x,d\mu(x)\rangle+\int_{\mathbb{R}^n}\mathrm{tr}(dH).
\end{equation*}
Our result, Theorem \ref{thm:positive}, follows immediately by taking $H=0$ in (\ref{eqn:polari}) and in (\ref{eqn:def}).

The representation formula of Cavalletti and Westdickenberg has found several applications to the compressible Euler equations; see \cite{Westdickenberg1, Westdickenberg2}. 

Let us note a related study of monotone maps by Alberti and Ambrosio \cite{Alberti} and works of Lions \cite{Lions} and of Carlier and Lachand-Robert \cite{Carlier} on representation of the polar cone to convex functions.

Let us mention also a paper by Galichon and Ghoussoub \cite{Galichon2}, which extends a result of Krauss \cite{Krauss}. The latter states that a map $u\colon\mathbb{R}^n\to\mathbb{R}^n$ is monotone if and only if there exists an antisymmetric concave-convex $H\colon\mathbb{R}^n\times\mathbb{R}^n\to\mathbb{R}$ such that $u$ belongs to the subdifferential, in the second variable, of $H$. In the paper \cite{Galichon2} a representation of this kind is extended to  jointly $N$-cyclically monotone families of maps. Also a variational formula for such families is provided. It is shown in particular that monotone maps are precisely the polar cone of the tangent cone at identity of the involutive measure preserving maps.

\subsection{Polar factorisation}

The connections of the study of polar cones and the optimal transport problems have been apparent since seminal work of Brenier \cite{Brenier2, Brenier} and subsequent developments of Caffarelli \cite{Caffarelli3,Caffarelli2, Caffarelli4}, Gangbo \cite{Gangbo3} and McCann \cite{McCann}. In \cite{Brenier} it is shown that any non-degenerate map $u\in L^p(X,\mu,\mathbb{R}^n)$, where $\mu$ is a probability measure on $X$, can be uniquely factorised 
\begin{equation*}
u=D\psi\circ s
\end{equation*}
into composition of derivative of a convex function $\psi\colon \Omega\to\mathbb{R}$ and a measure preserving map $s\colon X\to\Omega$, where $\Omega\subset\mathbb{R}^n$ is a bounded domain equipped with normalised Lebesgue measure. This result generalises, among others, the polar factorisation of matrices and the Helmholtz decomposition of vector fields.

Suppose that $\Omega=X$. Consider the polar cone $K$ to the set $S$ of measure preserving maps at the identity map. It may be proven, with the help of a result of Rockafellar \cite{Rockafellar}, that $K$ is precisely the set of gradients of convex functions on $\Omega$, see \cite{Brenier}. On the other hand, the polar cone of $S$ at the identity is the set of all maps such that the projection onto the set $S$ is the identity. If $S$ were a group, then any element could be written as an element of $K$ composed with an element of $S$, thus completing the proof. As $S$ is not a group, a more involved approach is used and the Monge--Kantorovich problems comes into play.

An elementary approach to the polar factorisation, basing merely on the dual problem, is presented in \cite{Gangbo3} and in \cite{Caffarelli3}.

In \cite{McCann} another proof via a geometrical approach is presented. It relies on the cyclical monotonicity introduced by Rockafellar in \cite{Rockafellar}.

In \cite{Caffarelli3, Caffarelli2, Caffarelli4} the regularity properties of the Brenier's solutions are studied.

\subsection{Divergence-measure fields}

Let us note that matrix-valued measures, that are under consideration here, belong to a class that extends the notion of divergence-measure fields to matrix-valued fields; see work of Anzellotii \cite{Anzellotti} for the first appearance of the notion in the literature and works of Chen, Comi, Frid, Torres \cite{GChen5, GChen2, GChen1, GChen3, GChen4} and references therein for more recent studies. Divergence-measure fields are vector fields whose divergences are measures. They have found many connections with and applications to non-linear conservation laws, including the new notions of normal traces, product rules, and Gauss-Green formulae.

The non-linear conservation laws that are considered in the papers cited above are the partial differential equations of the form
\begin{equation*}
\frac{\partial u}{\partial t}+divf(u)=0,
\end{equation*}
where $u\colon \mathbb{R}\times\mathbb{R}^n\to\mathbb{R}^m$ and $f\colon\mathbb{R}^m\to\mathbb{R}^{m\times n}$.

In the theory of such equations one considers so-called entropy solutions, which are weak solutions $u$ such that the Lax entropy inequality
\begin{equation}\label{eqn:sch}
\frac{\partial \eta(u)}{\partial t}+divq(u)\leq 0
\end{equation}
is satisfied for any convex entropy pair ($\eta$, $q$), i.e. a convex function $\eta\colon \mathbb{R}^m\to\mathbb{R}$ and a map $q\colon \mathbb{R}^m\to\mathbb{R}^{ n}$ such that $Dq=D\eta Df$.

Then (\ref{eqn:sch}) and the Riesz' theorem implies that the distribution $\frac{\partial \eta(u)}{\partial t}+divq(u)$ is a non-positive measure, so that the $(m+1)$-dimensional divergence of $(\eta(u),q(u))$ is a measure. This is to say, $(\eta(u),q(u))$ is a divergence-measure field.

\subsection{Further research}

Our interests in this problem was stimulated from another direction: a localisation technique in convex geometry. Klartag \cite{Klartag} has established a relation of the technique to the Monge--Kantorovich problem and has asked whether a generalisation to multiple constraints is conceivable by generalisation of the problem to vector measures. In \cite{Ciosmak3} certain aspects of the Klartag's approach has been successfully extended to this setting. However, in \cite{Ciosmak2} it has been proven that the so-called mass balance condition does not hold in the multi-dimensional context. In \cite{Ciosmak} it has been shown that a generalisation of the method used to prove the mass balance condition in the scalar case does not work in the vector case. Therefore, another approach is needed. The idea is to follow an approach to the Monge--Kantorovich problem present in the work of Evans and Gangbo \cite{Evans-Gangbo}.

Suppose that we are given $\mu\in\mathcal{M}(\mathbb{R}^n,\mathbb{R}^m)$ with $\mu(\mathbb{R}^n)=0$ and with finite first moments. Let $M_0\in\mathcal{M}(\mathbb{R}^n,\mathbb{R}^{m\times n})$ be an optimal measure in (\ref{eqn:measure}) and let $u_0\colon\mathbb{R}^n\to\mathbb{R}^m$ be an optimal $1$-Lipschitz map in (\ref{eqn:lipschitz}), see also Theorem \ref{thm:duality3}.

Suppose that $u_0$ is differentiable $\norm{M_0}_1$-almost everywhere, or less generally, that $M_0$ is absolutely continuous.
Define $dA=Du_0^*dM_0$. 
We shall call $A$ the transport density. 
Observe that it satisfies
\begin{align}
A\text{ is symmetric, positive semidefinite},\label{eqn:positive}\\
-div Du_0A=\mu,\label{eqn:di}\\
\norm{Du_0}\leq 1,\label{eqn:lip}\\
Du_0^*\text{ is an isometry on }\mathrm{im}\frac{dM_0}{d\norm{M_0}_1}.\label{eqn:isom}
\end{align}
This is a direct generalisation of the Monge--Kantorovich system, see \cite[Chapter 4]{Santambrogio}, considered frequently in the shape optimisation problems, see \cite{Bouchitte2, Bouchitte1}.
Indeed, by Theorem \ref{thm:duality3}, there is
\begin{equation*}
\Big\langle Du_0,\frac{dM_0}{d\norm{M_0}_1}\Big\rangle=1, \text{ almost everywhere.}
\end{equation*}
By equality case in matrix H\"older's inequality, Theorem \ref{thm:equality}, (\ref{eqn:positive}) holds true. By Remark \ref{rem:iso}, (\ref{eqn:isom}) holds true. By (\ref{eqn:isom}), equation (\ref{eqn:di}) is equivalent to $-divM_0=\mu$ and (\ref{eqn:lip}) is equivalent to $u_0$ being $1$-Lipschitz.

In the case of probability measures such transport density has been employed in \cite{Evans-Gangbo} to construct the first solution to the Monge--Kantorovich problem with metric cost.

The question that arises is under what conditions on a vector measure one may construct a corresponding optimal transport, in the sense of \cite{Ciosmak2}, see also Remark \ref{rem:Kanto}.

Note that, in the case of optimal transport of vector measures, there might exist no optimal transport for a given measure, as proven in \cite{Ciosmak2}.

In \cite[Theorem 4.16]{Santambrogio} it is shown that if a given measure is absolutely continuous, then so is the associated transport density. An open question is whether it  is true in the case of vector measures.
An equivalent problem is whether Theorem \ref{thm:dualityLebesgue} holds true for $q=1$, $\Omega=\mathbb{R}^n$.

If we knew that $Df_0^*$ is an isometry on $\mathrm{im}H_0$, as it happens if $f_0,H_0$ are optimisers for $q=1$, see Theorem \ref{thm:duality3}, then, after rescaling, see  (\ref{eqn:optifun}),
\begin{equation*}
-div\big( Df_0\abs{Df_0}^p\big)=h.
\end{equation*}
This is an analogue of the $(p+2)$-Laplacian equation for $f_0$, which is extensively used in \cite{Evans-Gangbo} to construct an absolutely continuous solution of the Monge--Kantorovich system and, in consequence, an optimal transport map. 
An interesting problem that arises is to determine whether the strategy used in \cite{Evans-Gangbo} conveys to the setting under consideration.

%
%
%

\subsection*{Outline of the article}

Section \ref{s:inequality} is devoted to a proof of matrix H\"older's inequality (Theorem \ref{thm:holder}) and Section \ref{s:equality} characterises equality cases (Theorem \ref{thm:equality} and Theorem \ref{thm:equalitypq}).

Section \ref{s:dual} is devoted to a proof of the duality theorem for optimal transport of vector measures in its divergence formulation (Theorem \ref{thm:duality3}).
In Section \ref{s:abs} we deal with duality for absolutely continuous measures (Theorem \ref{thm:dualityLebesgue}).

In Section \ref{s:polar} we employ results of the previous sections in order to provide characterisation of the dual cone to monotone maps (Theorem \ref{thm:positive}), thus reproving the result of \cite{Cavalletti4}.
In Section \ref{s:tangent} we generalise the result of Section \ref{s:polar} and provide a representation formula for polar cones to tangent cones of the unit ball of $\mathcal{C}^1(\mathbb{R}^n,\mathbb{R}^m)$ (Proposition \ref{pro:gene} and Theorem \ref{thm:cones}).
In Section \ref{s:tangentabs} we obtain another representation formulae for polar cones to the tangent cones of the unit ball of Sobolev space $\mathcal{W}^{1,p}(\Omega,\mathbb{R}^m)$ (Theorem \ref{thm:abso} and Theorem \ref{thm:polarsob}).

\section{Matrix H\"older's inequality}\label{s:mh}

In this section we consider matrix H\"older's inequality, which constitutes an extension of the classical H\"older's inequality to the matrix setting. The result of the section is a characterisation of equality cases.

\subsection{Inequality}\label{s:inequality}

Theorem \ref{thm:schwarz} and Theorem \ref{thm:holder} and their proofs are based on \cite{Baumgartner}. We refer the reader also to \cite{Bhatia}. We include the proofs for completeness and to provide an analysis of equality cases.

For a matrix $A$ of size $m\times n$ we shall denote by $\abs{A}$ its absolute value, that is
\begin{equation*}
|A|=(A^*A)^{\frac12}.
\end{equation*}

\begin{theorem}\label{thm:schwarz}
Let $A,B$ be two $m\times n$ matrices with real entries. Then
\begin{equation}\label{eqn:traces}
|\mathrm{tr}(A^*B)|\leq (\mathrm{tr} |A||B|)^{\frac12}\big(\mathrm{tr}|A^*||B^*|\big)^{\frac12}.
\end{equation}
\end{theorem}

\begin{lemma}\label{lem:basis}
Let $A$ be an $m\times n$ matrix with real entries. Then there exist orthonormal basis $(e_i)_{i=1}^n$ and $(f_j)_{j=1}^m$ and non-negative numbers $(a_i)_{i=1}^{n\wedge m}$ such that 
\begin{equation*}
\begin{aligned}
&Ae_i=a_if_i\text{ and }A^*f_j=a_je_j\text{ for all }i,j=1,\dotsc,n\wedge m\\
& \text{ and }Ae_i=0, A^*f_j=0\text{ for }i,j>n\wedge m.
\end{aligned}
\end{equation*}
\end{lemma}
\begin{proof}
For the proof, let $(e_i)_{i=1}^n$ be an orthonormal basis of eigenvectors of $A^*A$ with non-negative eigenvalues $(a_i^2)_{i=1}^n$, $a_i\geq 0$. We may assume that for $i\geq n\wedge m$ there is  $A^*Ae_i=0$, as the rank of $A^*A$ is at most $n\wedge m$. 

For indices such that $a_i=0$, we have $Ae_i=0$, as $\norm{Ae_i}^2=\langle A^*Ae_i,e_i\rangle=0$.
Set 
\begin{equation}\label{eqn:eigen}
f_j=\frac{1}{a_j}Ae_j\text{ for }j\text{ such that }a_j\neq 0.
\end{equation}
These are orthonormal, eigenvectors of $AA^*$ such that $A^*f_j=a_je_j$. 

Note that any eigenvector of $AA^*$ that is orthogonal to all $f_j$ has eigenvalue equal to zero. Indeed, if $f$ is such an eigenvector, then for all $j$ 
\begin{equation*}
0=\langle f, Ae_j\rangle=\langle A^*f,e_j\rangle,
\end{equation*}
that is $A^*f=0$ and thus $AA^*f=0$.

We may thus complement the eigenvectors (\ref{eqn:eigen})  to a full orthonormal basis $(f_j)_{j=1}^m$ by introducing eigenvectors of $AA^*$ with zero eigenvalue. For such vectors we have $A^*f_j=0$. 

This completes the proof of the lemma.
\end{proof}

\begin{proof}[Proof of Theorem \ref{thm:schwarz}]
Take orthonormal basis $(e_i)_{i=1}^n$, $(f_j)_{j=1}^m$ and non-negative numbers $(a_i)_{i=1}^{n\wedge m}$ for $A$ and $(g_i)_{i=1}^n$, $(h_j)_{j=1}^m$, $(b_i)_{i=1}^{n\wedge m}$ for $B$, as in Lemma \ref{lem:basis}.
Then
\begin{equation*}
|\mathrm{tr}(A^*B)|=\Big|\sum_{i,j=1}^{n\wedge m}a_ib_j\langle e_i,g_j\rangle \langle f_i, h_j\rangle\Big|.
\end{equation*}
The Cauchy--Schwarz inequality yields
\begin{equation*}
|\mathrm{tr}(A^*B)|\leq \Big(\sum_{i,j=1}^{n\wedge m}a_ib_j\langle e_i,g_j\rangle^2 \Big)^{\frac12}\Big(\sum_{i,j=1}^{n\wedge m}a_ib_j\langle f_i,h_j\rangle^2\Big)^{\frac12}.
\end{equation*}
Note now that for $i,j\leq n\wedge m$ we have
\begin{equation*}
|A|e_i=a_ie_i, |A^*|f_i=a_if_i\text{ and }|B|g_j=b_jg_j, |B^*|h_j=b_jh_j,
\end{equation*}
and for $i,j>n\wedge m$ 
\begin{equation*}
|A|e_i=0, |A^*|f_i=0\text{ and }|B|g_j=0, |B^*|h_j=0.
\end{equation*}
Therefore
\begin{equation*}
 |\mathrm{tr}(A^*B)|\leq (\mathrm{tr}(|A||B|)^{\frac12}(\mathrm{tr}(|A^*||B^*|)^{\frac12}.
\end{equation*}
\end{proof}

\begin{remark}\label{rem:eqhol}
The equality holds in the above inequality if and only if 
\begin{equation*}
(\langle e_i,g_j\rangle-\alpha \langle f_i,h_j\rangle)a_ib_j=0
\end{equation*}
for some constant $\alpha$ and all indices $i,j$.
\end{remark}

Below $\norm{\cdot}$ denotes the operator norm of a matrix, regarded as a linear operator between Euclidean spaces.

\begin{theorem}\label{thm:holder}
Let $A,B$ be two $m\times n$ matrices with real entries. Then
\begin{equation}\label{eqn:shat}
|\mathrm{tr}(A^*B)|\leq (\mathrm{tr} |A|^{p})^{\frac1p}(\mathrm{tr}(|B|^q)^{\frac1q}
\end{equation}
for all $p,q\in (1,\infty)$ such that $\frac1p+\frac1q=1$.
Moreover
\begin{equation}\label{eqn:oper}
|\mathrm{tr}(A^*B)|\leq \mathrm{tr} |A|\norm{B}.
\end{equation}
\end{theorem}
\begin{proof}
Let $p,q\in (1,\infty)$ be such that $\frac1p+\frac1q=1$. Using the notation from the above theorem, let us note that
\begin{equation*}
\mathrm{tr}|A|^p=\mathrm{tr}|A^*|^p=\sum_{i=1}^{n\wedge m}a_i^p.
\end{equation*}
Hence, for the proof it is enough to show that every factor on the right-hand side of the inequality (\ref{eqn:traces}) is bounded above by $(\mathrm{tr} |A|^{p})^{\frac1p}(\mathrm{tr}(|B|^q)^{\frac1q}$. For this, observe that, thanks to orthonormality of the basis,
\begin{equation*}
\sum_{i=1}^n\langle e_i,g_j\rangle^2=1,
\end{equation*}
for all $j=1,\dotsc,m$ and that
\begin{equation*}
\sum_{j=1}^m\langle e_i,g_j\rangle^2=1,
\end{equation*}
for all $i=1,\dotsc,n$.
Therefore, using H\"older's inequality, we get
\begin{align*}
\mathrm{tr}(|A||B|)&=\sum_{i,j=1}^{n\wedge m}a_ib_j\langle e_i,g_j\rangle^2\leq \Big(\sum_{i,j=1}^{n\wedge m}a_i^p\langle e_i,g_j\rangle^2\Big)^{\frac1p}\Big(\sum_{i,j=1}^{n\wedge m}b_j^q\langle e_i,g_j\rangle^2\Big)^{\frac1q}=\\
&=\Big(\sum_{i,j=1}^{n\wedge m}a_i^p\Big)^{\frac1p}\Big(\sum_{i,j=1}^{n\wedge m}b_j^q\Big)^{\frac1q}=(\mathrm{tr}|A|^p)^{\frac1p}(\mathrm{tr}|B|^q)^{\frac1q}
\end{align*}
Proceeding analogously for $\mathrm{tr}(|A^*||B^*|)$ we get the desired inequality.

For the second part of the theorem observe that
\begin{equation*}
\norm{B}=\norm{B^*}=\max\{b_i\mid i=1,\dotsc,n\wedge m \}
\end{equation*}
and that
\begin{equation*}
\mathrm{tr}|A|=\mathrm{tr}|A^*|=\sum_{i,j=1}^{n\wedge m}a_i.
\end{equation*}
Therefore
\begin{equation*}
\mathrm{tr}(|A||B|)=\sum_{i,j=1}^{n\wedge m}a_ib_j\langle e_i,g_j\rangle^2\leq \norm{B}\mathrm{tr}|A|.
\end{equation*}
Proceeding analogously for $\mathrm{tr}(|A^*||B^*|)$ we get the desired inequality.
\end{proof}

\begin{remark}\label{rem:eqholp}
If $p,q\in(1,\infty)$, then the equality in inequality (\ref{eqn:shat}) in the above theorem holds true if and only if there exists a constant $\beta$ such that
\begin{equation}\label{eqn:equalty0}
(a_i^p-\beta b_j^q)\langle e_i, g_j\rangle=0\text{ for all }i=1,\dotsc, n\text{ and }j=1,\dotsc,m
\end{equation}
and a constant $\alpha$ such that
\begin{equation}\label{eqn:equality}
(\langle e_i,g_j\rangle-\alpha \langle f_i,h_j\rangle)a_ib_j=0\text{ for all }i=1,\dotsc, n\text{ and }j=1,\dotsc,m.
\end{equation}
For the inequality (\ref{eqn:oper}) in the theorem, there holds equality if and only if (\ref{eqn:equality}) is satisfied and moreover
\begin{equation}\label{eqn:equality2}
a_i(b_j-b)\langle e_i, g_j\rangle=0\text{ for all }i=1,\dotsc, n\text{ and }j=1,\dotsc,n
\end{equation}
and some number $b$.
\end{remark}

\subsection{Equality cases}\label{s:equality}

\begin{example}\label{exa:identity}
Suppose that $n=m$ and that $B$ is the identity matrix. Then the inequality (\ref{eqn:oper}) yields that for all $n\times n$ matrices $A$ there is
\begin{equation*}
|\mathrm{tr}A|\leq \mathrm{tr}|A|.
\end{equation*}
Equality here holds if and only if $A$ is symmetric and semi-definite. Indeed, it holds if and only if (\ref{eqn:equality}) and (\ref{eqn:equality2}) are satisfied. Note that (\ref{eqn:equality2}) for $B=\mathrm{Id}$ clearly holds true. Condition (\ref{eqn:equality}) holds if and only if 
\begin{equation*}
(\langle e_i, e_j\rangle-\alpha \langle f_i,e_j\rangle )a_i=0\text{ for all }i=1,\dotsc,n.
\end{equation*}
Here we took $g_j=h_j=e_j$.
Equivalently, for any $i$ such that $a_i\neq 0$ and all $j=1,\dotsc,n$
\begin{equation*}
\langle e_i-\frac{\alpha}{a_i}Ae_i,e_j\rangle=0.
\end{equation*}
That is
\begin{equation*}
Ae_i=\frac{a_i}{\alpha}e_i.
\end{equation*}
If $a_i=0$, then $Ae_i=0$. We see thus that $A$ is diagonal with  eigenvalues of fixed sign and thus it is symmetric and semi-definite. The converse implication is obvious.
\end{example}

\begin{definition}
For $p\in [1,\infty]$, the quantity $(\mathrm{tr}|A|^p)^{\frac1p}$ is called the Schatten $p$-norm of a matrix $A$. We shall denote it by $\norm{A}_p$. We shall write $\langle A, B\rangle = \mathrm{tr}(AB^*)$.
\end{definition}

Let us now analyse carefully what the conditions (\ref{eqn:equality}) and (\ref{eqn:equality2}) mean.

\begin{theorem}\label{thm:equality}
Suppose that $A,B$ are $m\times n$ matrices such that $\norm{B}\leq 1$. Then the condition
\begin{equation}\label{eqn:trace}
\mathrm{tr}(A^*B)=\mathrm{tr}(|A|)
\end{equation}
holds if and only if $A^*B$ is symmetric and positive semi-definite and 
\begin{equation}\label{eqn:four}
A^*BB^*A=A^*A.
\end{equation}
Moreover, if $A\neq 0$, then $\norm{B}=1$.
\end{theorem}
\begin{proof}
If $A=0$, then the equivalence clearly holds true. Suppose that $A\neq 0$ and that (\ref{eqn:trace}) holds true. Then, by (\ref{eqn:oper}), we have
\begin{equation*}
\mathrm{tr}(|A|)=\mathrm{tr}(A^*B)\leq \mathrm{tr}(|A|)\norm{B}\leq \mathrm{tr}(|A|).
\end{equation*}
Therefore the equality (\ref{eqn:trace}) holds if and only if the conditions (\ref{eqn:equality}) and (\ref{eqn:equality2}) are satisfied and $\norm{B}=1$. 

Suppose that these conditions hold true. Note that if $\alpha=0$ in (\ref{eqn:equality}), then we would have $\mathrm{tr}(A^*B)=0$, contrary to the assumptions. Therefore $\alpha\neq 0$.  
Condition (\ref{eqn:equality2}) may be equivalently stated as 
\begin{equation*}
(b_j-b)g_j\in\mathrm{ker}A\text{ for all }j=1,\dotsc, m.
\end{equation*}
Indeed, we may write for $i=1,\dotsc, n$ and $j=1,\dotsc, m$
\begin{equation*}
0=a_i(b_j-b)\langle e_i,g_j\rangle=\langle A^*f_i,(b_j-b)g_j\rangle=\langle f_i, A(b_j-b)g_j\rangle.
\end{equation*}
Thus we have $A(b_j-b)g_j=0$. 

Condition (\ref{eqn:equality}), with the above observation, implies that
\begin{align*}
0&=(\langle e_i,g_j\rangle -\alpha\langle f_i,h_j\rangle)a_ib_j=\langle A^*f_i, b_jg_j\rangle-\alpha \langle Ae_i,Bg_j\rangle=\\
&=\langle f_i, Ab_jg_j\rangle-\alpha \langle B^*Ae_i,g_j\rangle=\langle f_i,Abg_j\rangle -\alpha \langle B^*Ae_i,g_j\rangle=\\
&=\langle ba_ie_i,g_j\rangle-\alpha\langle B^*Ae_i,g_j\rangle.
\end{align*}
It follows that for $i=1,\dotsc,n$
\begin{equation*}
B^*Ae_i=a_i\frac{b}{\alpha}e_i.
\end{equation*}
Thus $B^*A$ is symmetric and semi-definite. Hence 
\begin{equation}\label{eqn:cosikp}
\mathrm{tr}(|A|)=\mathrm{tr}(A^*B)=\mathrm{tr}(|B^*A|)=\frac{b}{|\alpha|}\mathrm{tr}(|A|).
\end{equation}
This is to say, $\frac{b}{|\alpha|}=1$. It follows that 
\begin{equation*}
A^*BB^*A=A^*A.
\end{equation*}
Observe also that $B^*A$ is positive semi-definite, as the quantities in (\ref{eqn:cosikp}) are non-negative.

Conversely, if $\norm{B}=1$, $A^*B$ is symmetric and positive semi-definite and 
\begin{equation*}
A^*BB^*A=A^*A, 
\end{equation*}
then
\begin{equation*}
|A|=|A^*B|=A^*B
\end{equation*}
and thus
\begin{equation*}
\mathrm{tr}(A^*B)=\mathrm{tr}(|A|).
\end{equation*}
\end{proof}

\begin{corollary}\label{col:diagonalise}
Suppose that $A,B$ are as above. Then
\begin{equation*}
B^*A=A^*B, BA^*=AB^*
\end{equation*}
are positive semi-definite and
\begin{equation*}
A^*BB^*A=A^*A, AB^*BA^*=AA^*.
\end{equation*}
Moreover $A^*A$ and $B^*B$ diagonalise in a common orthonormal basis.
\end{corollary}
\begin{proof}
Observe that $\mathrm{tr}(A^*B)=\mathrm{tr}(AB^*)$ and that $\mathrm{tr}(|A|)=\mathrm{tr}(|A^*|)$. Thus the assertion follows by an application of Theorem \ref{thm:equality}. The fact that $A^*A$ and $B^*B$ diagonalise in a common orthonormal basis is a consequence of
\begin{equation*}
A^*AB^*B=B^*BA^*A.
\end{equation*}
\end{proof}

\begin{remark}\label{rem:iso}
Note that condition (\ref{eqn:four}) takes the form
\begin{equation*}
BB^*=\mathrm{Id},
\end{equation*}
provided that $A$ is invertible. This is to say, $B$ is then an isometry.
If $A$ is not invertible, in particular if $m\neq n$, then $B^*$ is an isometry, if restricted to $\mathrm{im}A$. Indeed, for $x=Ay$, we have
\begin{equation*}
\norm{B^*x}^2=\langle BB^*x,x\rangle=\langle A^*BB^*Ay,y\rangle=\langle A^*Ay,y\rangle=\norm{x}^2.
\end{equation*}
Conversely, if $BB^*=\mathrm{Id}$ on $\mathrm{im}A$, then clearly $A^*BB^*A=A^*A$.  Thus, in Theorem \ref{thm:equality}, we might write that the equivalent conditions are: $B^*A$ is symmetric and positive semi-definite and that $B^*$ is an isometry on $\mathrm{im}A$. 
\end{remark}

\begin{theorem}\label{thm:equalitypq}
Suppose that $m\leq n$ and that $A,B$ are $m\times n$ matrices, $A,B\neq 0$. Let $p,q\in(1,\infty)$ be such that $\frac1p+\frac1q=1$. Then the equality
\begin{equation}\label{eqn:trpq}
\mathrm{tr}(A^*B)=(\mathrm{tr}(|A|^p))^{\frac1p}(\mathrm{tr}(|B|^q))^{\frac1q}
\end{equation}
holds if and only if $B^*A$ is symmetric and positive semi-definite and for some $c>0$,
\begin{equation}\label{eqn:ce}
A^*B=B^*A=c^p(A^*A)^{\frac{p}2}=c^{-q}(B^*B)^{\frac{q}2}.
\end{equation}
Equivalently,
\begin{equation}\label{eqn:equic}
\frac{A^*B}{(\mathrm{tr}(|A|^p))^{\frac1p}(\mathrm{tr}(|B|^q))^{\frac1q}}=\frac{B^*A}{(\mathrm{tr}(|A|^p))^{\frac1p}(\mathrm{tr}(|B|^q))^{\frac1q}}=\frac{(A^*A)^{\frac{p}2}}{\mathrm{tr}(|A|^p)}=\frac{(B^*B)^{\frac{q}2}}{\mathrm{tr}(|B|^q)}.
\end{equation}
\end{theorem}
\begin{proof}
Let us recall -- see (\ref{eqn:equalty0}) and (\ref{eqn:equality}) --  that the equality holds if and only if
\begin{equation}\label{eqn:pq}
(a_i^p-\beta b_j^q)\langle e_i,g_j\rangle=0\text{ for all }i=1,\dotsc,n, j=1,\dotsc,m,
\end{equation}
and
\begin{equation}\label{eqn:11}
(\langle e_i,g_j\rangle-\alpha \langle f_i,h_j\rangle )a_ib_j=0\text{ for all }i=1,\dotsc,n, j=1,\dotsc,m.
\end{equation}
Form (\ref{eqn:pq}) it follows that $\beta\neq 0$. Otherwise $A=0$. Form (\ref{eqn:11}) it follows that $\alpha\neq 0$, otherwise $A^*B=0$. Then, as there is equality in (\ref{eqn:trpq}), we must have $A=0$ or $B=0$, contrary to the assumptions. In what follows we assume therefore that $\alpha\neq 0, \beta\neq 0$.
If $\langle e_i,g_j\rangle\neq 0$, then, by (\ref{eqn:pq}),
\begin{equation*}
b_j=\frac{1}{\beta^{\frac1q}}a_i^{\frac{p}q}
\end{equation*}
and thus, by (\ref{eqn:11}),
\begin{equation*}
\Bigg\langle \frac{a_i^{1+\frac{p}q}}{\alpha \beta^{\frac1q}}e_i-B^*Ae_i,g_j\Bigg\rangle=0.
\end{equation*}
If $\langle e_i,g_j\rangle=0$ then the above formula holds as well, by (\ref{eqn:11}). We infer that for $i=1,\dotsc,n$
\begin{equation*}
B^*Ae_i=\frac{a_i^{1+\frac{p}q}}{\alpha \beta^{\frac1q}}e_i.
\end{equation*}
Analogously we have for $j=1,\dotsc,n$
\begin{equation*}
A^*Bg_j=\frac{\beta^{\frac1p} b_j^{1+\frac{q}p}}{\alpha} g_j.
\end{equation*}
Equivalently
\begin{equation*}
B^*A=\frac1{\alpha \beta^{\frac1q}}(A^*A)^{\frac{p}2}\text{ and }A^*B=\frac{\beta^{\frac1p}}{\alpha}(B^*B)^{\frac{q}2}.
\end{equation*}
It follows that $B^*A$ is symmetric and semi-definite. By (\ref{eqn:trpq}) it is positive semi-definite, so $\alpha>0$, and therefore 
\begin{equation}\label{eqn:check}
\mathrm{tr}(B^*A)=(\mathrm{tr}(B^*A))^{\frac1p}(\mathrm{tr}(A^*B))^{\frac1q}
=\frac1{\alpha}(\mathrm{tr}|A|^p)^{\frac1p}(\mathrm{tr}|B|^q)^{\frac1q}.
\end{equation}
Thus $\alpha=1$ and the equality (\ref{eqn:ce}) holds with $c=\beta^{-\frac1{pq}}$.
Converse implication follows readily, as in (\ref{eqn:check}). Condition (\ref{eqn:equic}) follows from (\ref{eqn:ce}) by taking traces and computing the number $c$.
\end{proof}

\begin{remark}
The following duality formula holds true
\begin{equation*}
\norm{A}_p=\sup\{\mathrm{tr}(AB^*)\mid  \norm{B}_q\leq 1\},
\end{equation*}
where $\frac1p+\frac1q=1$, $p,q\in [1,\infty]$. 
Indeed, take any matrix $A$ and its polar factorisation $A=UD$, where $U$ is a partial isometry and $D$ is symmetric and positive semi-definite such that $U^*U$ is projection onto $\mathrm{im}D$. Set $B=UD^{\frac{p}{q}}$. Then
\begin{equation*}
\norm{A}_p=\norm{D}_p, \norm{B}_q=\norm{D^{\frac{p}{q}}}_q=\norm{D}_p^{\frac{p}q},
\end{equation*}
and 
\begin{equation*}
\langle A,B\rangle=\mathrm{tr}(B^*A)=\mathrm{tr}(D^{1+\frac{p}{q}})=\mathrm{tr}(D^p)=\norm{A}_p\norm{B}_q.
\end{equation*}
If $p=1$, we take $B=U$. If $p=\infty$, let $v$ be a vector such that $Dv=\lambda v$ with $\lambda$ such that $|\lambda|=\norm{D}$.  Take $B=UD'$, where $D'=\frac{\lambda}{\abs{\lambda}} vv^*$.
\end{remark}

\section{Divergence formulation of optimal transport of vector measures}

Here we introduce the aforementioned formulation of optimal transport of vector measures. The following material deals also with absolutely continuous measures. 

\subsection{Duality formula for optimal transport of vector measures}\label{s:dual}

Suppose we are given a Borel vector-valued measure $\mu\in\mathcal{M}(\mathbb{R}^n,\mathbb{R}^m)$ of finite total variation, such that $\mu(\mathbb{R}^n)=0$ and $\mu$ has finite first moments, i.e.
\begin{equation*}
\int_{\mathbb{R}^n}\norm{x}d\norm{\mu}(x)<\infty.
\end{equation*}
Here $\norm{\mu}$ denotes the total variation of $\mu$. In what follows we shall consider the Banach space $\mathcal{C}^1(\mathbb{R}^n,\mathbb{R}^m)$ of continuously differentiable functions with bounded derivative and with a seminorm given by
\begin{equation*}
\norm{u}_{\mathcal{C}^1}=\sup\Big\{\norm{Du(x)}\mid x\in\mathbb{R}^n\Big\}.
\end{equation*}
The seminorm induces a norm on the subspace of functions vanishing at the origin.
By $\mathcal{C}^1_0(\mathbb{R}^n,\mathbb{R}^m)$ we denote the subspace of functions with derivative vanishing at infinity.

\begin{definition}
Suppose that $\mu\in\mathcal{M}(\mathbb{R}^n,\mathbb{R}^m)$ has finite first moments and $\mu(\mathbb{R}^n)=0$. Define
\begin{equation*}
\Xi(\mu)=\{M\in\mathcal{M}(\mathbb{R}^n,\mathbb{R}^{m\times n})\mid  -divM=\mu\}.
\end{equation*}
Here we write 
\begin{equation*}
-div M=\mu
\end{equation*}
if
\begin{equation}\label{eqn:div}
\int_{\mathbb{R}^n}\langle f,d\mu\rangle=\int_{\mathbb{R}^n}\langle Df, dM\rangle
\end{equation}
for all functions $f\in\mathcal{C}^1_0(\mathbb{R}^n,\mathbb{R}^m)$.
For any $\mu$ as above define
\begin{equation*}
I(\mu)=\inf\{\norm{M}_{\mathcal{M}}\mid  M\in\Xi(\mu)\}
\end{equation*}
and
\begin{equation*}
J(\mu)=\sup\Big\{\int_{\mathbb{R}^n}\langle u,d\mu\rangle\mid u\in\mathcal{C}^1(\mathbb{R}^n,\mathbb{R}^m),  \norm{u}_{\mathcal{C}^1}\leq 1\Big\}.
\end{equation*}
\end{definition}
Above we consider the total variation norm $\norm{M}_{\mathcal{M}}$ of a matrix-valued measure $M$ with respect to the Schatten $1$-norm. Total variation measure of $M$ with respect to the Schatten $1$-norm shall be denoted by $\norm{M}_1$. Recall that the Schatten $1$-norm of $A\in\mathbb{R}^{m\times n}$ is defined via the formula
\begin{equation*}
\norm{A}_1=\mathrm{tr}(A^*A)^{\frac12}.
\end{equation*}

\begin{lemma}\label{lem:div}
Suppose that $\mu\in\mathcal{M}(\mathbb{R}^n,\mathbb{R}^m)$ is such that $\mu(\mathbb{R}^n)=0$ and has finite first moments. Let $M\in\mathcal{M}(\mathbb{R}^n,\mathbb{R}^{m\times n})$ satisfy $-divM=\mu$. Then for all maps $f\in\mathcal{C}^1(\mathbb{R}^n,\mathbb{R}^m)$ there is
\begin{equation*}
\int_{\mathbb{R}^n}\langle f,d\mu\rangle=\int_{\mathbb{R}^n}\langle Df,dM\rangle.
\end{equation*}
\end{lemma}
\begin{proof}
By the assumption, for any $f_0\in\mathcal{C}^1_0(\mathbb{R}^n,\mathbb{R}^m)$ there is
\begin{equation*}
\int_{\mathbb{R}^n}\langle f_0,d\mu\rangle=\int_{\mathbb{R}^n}\langle Df_0,dM\rangle.
\end{equation*}
Pick $f\in\mathcal{C}^1(\mathbb{R}^n,\mathbb{R}^m)$ and $\epsilon>0$. We may assume that $f(0)=0$, so that for $x\in\mathbb{R}^n$ there is
\begin{equation}\label{eqn:bo}
\norm{f(x)}\leq \norm{f}_{\mathcal{C}^1}\norm{x}.
\end{equation}
 Choose a compact ball $K\subset\mathbb{R}^n$ such that
\begin{equation}\label{eqn:tight}
\int_{K^c}\norm{x}d\norm{\mu}(x)<\epsilon \text{ and }\norm{M}(K^c)<\epsilon.
\end{equation}
Let $\phi\in \mathcal{C}^1_0(\mathbb{R}^n,\mathbb{R})$ be a non-negative function, equal to one on $K$, bounded above by one and such that $\norm{D\phi}(x)\leq\frac{1}{1+\norm{x}}$ for $x\notin K$.
Then $\phi f\in \mathcal{C}^1_0(\mathbb{R}^n,\mathbb{R}^m)$, so 
\begin{equation*}
\int_{\mathbb{R}^n}\langle\phi f,d\mu\rangle=\int_{\mathbb{R}^n}\langle\phi Df,dM\rangle+\int_{\mathbb{R}^n}\langle D\phi f,dM\rangle.
\end{equation*}
Therefore, taking into account that $\phi=1$ on $K$ and $D\phi=0$ on $K$, we get

\begin{align*}
&\int_{\mathbb{R}^n}\langle f,d\mu\rangle-\int_{\mathbb{R}^n}\langle Df,dM\rangle=\\
&=  \int_{K^c}(1-\phi)\langle f,d\mu\rangle-\int_{K^c}(1-\phi)\langle Df,dM\rangle+ \int_{K^c}\langle D\phi f,dM\rangle.
\end{align*}
The first two terms of the right-hand side of the above equality are bounded by $\epsilon \norm{f}_{\mathcal{C}^1}$ each, by (\ref{eqn:bo}) and (\ref{eqn:tight}). The third term is bounded by $\epsilon\norm{f}_{\mathcal{C}^1}$, by the assumption on $\phi$, by (\ref{eqn:bo}) and by (\ref{eqn:tight}). Since $\epsilon>0$ was arbitrary, the left-hand side of the equality is equal to zero. This completes the proof.
\end{proof}

\begin{theorem}\label{thm:duality3}
For any $\mu\in\mathcal{M}(\mathbb{R}^n,\mathbb{R}^m)$ such that $\mu(\mathbb{R}^n)=0$ and with finite first moments we have
\begin{equation*}
I(\mu)=J(\mu).
\end{equation*}
Moreover there exists a $1$-Lipschitz $u_0\colon\mathbb{R}^n\to\mathbb{R}^m$ such that
\begin{equation*}
J(\mu)=\int_{\mathbb{R}^n}\langle u_0,d\mu\rangle,
\end{equation*}
and there exists $M_0\in\Xi(\mu)$ such that
\begin{equation*}
I(\mu)=\norm{M_0}_{\mathcal{M}}.
\end{equation*}
Also,
\begin{equation}\label{eqn:jot}
J(\mu)=\sup\Big\{ \int_{\mathbb{R}^n}\langle f,d\mu\rangle\mid f\in\mathcal{C}^1_0(\mathbb{R}^n,\mathbb{R}^m), \norm{f}_{\mathcal{C}^1}\leq 1\Big\}.
\end{equation}
If $u_0,M_0$ are optimisers such that $u_0$ is differentiable $\norm{M_0}_1$-almost everywhere, then
\begin{equation}\label{eqn:optimal}
\Big\langle Du_0,\frac{dM_0}{d\norm{M_0}_1}\Big\rangle =1
\end{equation}
$\norm{M_0}_1$-almost everywhere.
\end{theorem}
\begin{proof}
On the subspace 
\begin{equation*}
\Big\{Df\mid f\in\mathcal{C}^1_0(\mathbb{R}^n,\mathbb{R}^m)\Big\}
\end{equation*}
of the space $\mathcal{C}_0(\mathbb{R}^n,\mathbb{R}^m)$ of continuous functions vanishing at infinity, define a functional $\Lambda$
by the formula
\begin{equation*}
\Lambda(Df)=\int_{\mathbb{R}^n}\langle f, d\mu\rangle.
\end{equation*}
As $\mu(\mathbb{R}^n)=0$ it is well defined. It is continuous as
\begin{equation*}
\Lambda(Df)=\int_{\mathbb{R}^n}\langle f(x)-f(0), d\mu(x)\rangle=\int_{\mathbb{R}^n}\int_{[0,1]}\langle Df(tx)(x),d\mu(x)\rangle d\lambda(t),
\end{equation*}
and thus $\norm{\Lambda}\leq \int_{\mathbb{R}^n}\norm{x}d\norm{\mu}(x)$.
By the Hahn--Banach theorem it follows that we may extend $\Lambda$ to a continuous linear functional on $\mathcal{C}_0(\mathbb{R}^n,\mathbb{R}^{m\times n})$ preserving its norm.
From Riesz' theorem it follows that the dual space of $\mathcal{C}_0(\mathbb{R}^n,\mathbb{R}^{m\times n})$ is isometrically isomorphic to the space of Borel measures $\mathcal{M}(\mathbb{R}^n,\mathbb{R}^{m\times n})$. We obtain a matrix-valued measure $M_0\in\mathcal{M}(\mathbb{R}^n,\mathbb{R}^{m\times n})$ such that
\begin{equation*}
\Lambda(Df)=\int_{\mathbb{R}^n}\langle Df,dM_0\rangle
\end{equation*}
and such that $\norm{M_0}_{\mathcal{M}}=\norm{\Lambda}$. Moreover 
\begin{align*}
\norm{\Lambda}&=\sup\Big\{\Lambda(Df)\mid f\in \mathcal{C}_0^1(\mathbb{R}^n,\mathbb{R}^m), \norm{f}_{\mathcal{C}^1}\leq 1\Big\}=\\&=\sup\Big\{ \int_{\mathbb{R}^n}\langle f,d\mu\rangle\mid f\in\mathcal{C}_0^1(\mathbb{R}^n,\mathbb{R}^m), \norm{f}_{\mathcal{C}^1}\leq 1\Big\}\leq J(\mu).
\end{align*}
Observe that, by Theorem \ref{thm:holder} and Lemma \ref{lem:div},
\begin{equation}\label{eqn:ineq}
J(\mu)\leq I(\mu).
\end{equation}
As we have just found a measure $M_0$ with $\norm{M_0}_{\mathcal{M}}\leq J(\mu)$ we see that there holds equality in (\ref{eqn:ineq}).

We shall now show that there exists a $1$-Lipschitz map $u_0$ such that
\begin{equation*}
\int_{\mathbb{R}^n}\langle u_0,d\mu\rangle=\sup\Big\{\int_{\mathbb{R}^n}\langle u,d\mu\rangle\mid f\in\mathcal{C}^1(\mathbb{R}^n,\mathbb{R}^m), \norm{f}_{\mathcal{C}^1}\leq 1\Big\}.
\end{equation*}
For this, we may use the Arzel\`a--Ascoli theorem. Choose a sequence 
\begin{equation*}
(u_k)_{k=1}^{\infty}\in \mathcal{C}^1(\mathbb{R}^n,\mathbb{R}^m), 
\end{equation*}
with $\norm{u_k}_{\mathcal{C}^1}\leq 1$, $u_k(0)=0$ and such that
\begin{equation*}
\lim_{k\to\infty}\int_{\mathbb{R}^n}\langle u_k,d\mu\rangle=J(\mu).
\end{equation*}
We may extract a subsequence $(u_{k_l})_{l=1}^{\infty}$ that converges locally uniformly to a $1$-Lipschitz function $u_0$. We claim that 
\begin{equation}\label{eqn:eqn}
\int_{\mathbb{R}^n}\langle u_0,d\mu\rangle=J(\mu).
\end{equation}
Choose $\epsilon>0$ and a compact set $K\subset\mathbb{R}^n$ such that 
\begin{equation*}
\int_{K^c}\norm{x}d\norm{\mu}(x)<\epsilon.
\end{equation*}
If $\norm{u_{k_l}(x)-u_0(x)}<\epsilon$ for all $l\geq N$ and all $x\in K$, then
\begin{equation*}
\Big | \int_{\mathbb{R}^n}\langle u_0,d\mu\rangle-\int_{\mathbb{R}^n}\langle u_{k_l},d\mu\rangle\Big | \leq 2 \int_{K^c}\norm{x}d\norm{\mu}(x)+\epsilon \norm{\mu}(K).
\end{equation*}
Letting $l$ tend to infinity and then $\epsilon$ to zero we get (\ref{eqn:eqn}).

For the last part of the theorem observe that if $-divM=\mu$ and $u_0$ is differentiable $\norm{M}_1$-almost everywhere, with $\norm{Du_0}$ bounded above by one, $\norm{M}_1$-almost everywhere, then by matrix H\"older's inequality (see Section \ref{s:mh})
\begin{equation*}
\Big\langle Du_0,\frac{dM}{d\norm{M}_1}\Big\rangle \leq 1.
\end{equation*}
If $M_0$ and $u_0$ are the optimisers then integrating the inequality with respect to $d\norm{M_0}_1$, we would get $J(\mu)\leq I(\mu)$. Hence equality holds if and only if (\ref{eqn:optimal}) holds.
\end{proof}

\begin{remark}\label{rem:Kanto}
The problem of finding optimal measure $M_0$ is a relaxation of the optimal transport of vector measures, as formulated in \cite{Ciosmak2}. In this problem there holds the Kantorovich--Rubinstein duality. That is
\begin{equation*}
\inf\Big\{\int_{\mathbb{R}^n\times\mathbb{R}^n}\norm{x-y}d\norm{\pi}(x,y)\mid \mathrm{P}_1\pi-\mathrm{P}_2\pi=\mu\Big\},
\end{equation*}
where $\mathrm{P}_1\pi,\mathrm{P}_2\pi$ denote the first and the second marginal of $\pi$ respectively, is equal to 
\begin{equation*}
\sup\Big\{\int_{\mathbb{R}^n}\langle f,d\mu\rangle\mid f\colon\mathbb{R}^n\to\mathbb{R}^m\text{ is }1\text{-Lipschitz}\Big\}.
\end{equation*}
Using convolution with a smoothing kernel we may show that the right-hand side of the above equality is equal to 
\begin{equation*}
\sup\Big\{\int_{\mathbb{R}^n}\langle f,d\mu\rangle\mid f\in\mathcal{C}^1(\mathbb{R}^n,\mathbb{R}^m), \norm{f}_{\mathcal{C}^1}\leq 1\Big\}.
\end{equation*}
\end{remark}

\subsection{Absolutely continuous vector measures}\label{s:abs}

We shall be now interested in measures on open, bounded, connected subsets $\Omega\subset\mathbb{R}^n$ with Lipschitz boundary $\partial\Omega$. Let us now restrict to consideration of matrix-valued measures $M\in\mathcal{M}(\Omega,\mathbb{R}^{m\times n})$  that are absolutely continuous with respect to the Lebesgue measure $\lambda$. Let $p,q\in[1,\infty]$ be such that $\frac1p+\frac1q=1$. Let $h\in L^q(\Omega,\mathbb{R}^m)$
be such that
\begin{equation*}
\int_{\Omega}hd\lambda=0.
\end{equation*}
Let us define
\begin{equation*}
\Gamma_{\lambda}^q(h)=\Big\{H\in L^q(\Omega,\mathbb{R}^{m\times n})\mid -divH=h\Big\}.
\end{equation*}
For $H\in L^q(\Omega,\mathbb{R}^{m\times n})$ we say that $-divH=h$ with no-flux boundary conditions if
\begin{equation}\label{eqn:ediv}
\int_{\Omega}\langle Df,H\rangle d\lambda=\int_{\Omega}\langle f,h\rangle d\lambda,
\end{equation}
for all  functions $f\in\mathcal{W}^{1,p}(\Omega,\mathbb{R}^m)$. 
The above condition is equivalent to demanding that the negative divergence of $Hd\lambda$ is equal to $hd\lambda$ in the sense of Section \ref{s:dual} and that the rows of $H$ are tangent to the boundary of $\Omega$.

Likewise, we could consider condition (\ref{eqn:ediv}) for all functions $f\in\mathcal{W}_0^{1,p}(\Omega,\mathbb{R}^m)$. Then it would be equivalent to demanding that the negative divergence of $Hd\lambda$ is equal to $hd\lambda$ in the sense of Section \ref{s:dual}.

Here, $\mathcal{W}_0^{1,p}(\Omega,\mathbb{R}^m)$
denotes the closure of smooth, compactly supported functions in $\mathcal{W}^{1,p}(\Omega,\mathbb{R}^m)$.

Let us recall that $\mathcal{W}^{1,p}(\Omega,\mathbb{R}^m)$ denotes the Sobolev space of Borel measurable functions $f$ on $ \Omega$ that  admit a weak derivative $Df$ and with finite norm given by the formula
\begin{equation*}
\Big(\int_{\Omega} \mathrm{tr}(|Df|^p)d\lambda+\int_{\Omega}\norm{f}^pd\lambda\Big)^{\frac1p}
\end{equation*}
Recall that $\mathrm{tr}(|G|^p)^{\frac1p}$ for a matrix $G\in\mathbb{R}^{m\times n}$ is called the Schatten $p$-norm. 

Let us recall that the Poincar\'e inequality (see e.g. \cite{Evans2, Evans}) states that for any $p\in [1,\infty]$ there exists a constant $C$, which depends on $\Omega$ and $n,p$, such that for any map $f\in\mathcal{W}^{1,p}(\Omega,\mathbb{R}^m)$ there is
\begin{equation*}
\int_{\Omega}\Big\lVert f -\frac{\int_{\Omega} fd\lambda}{\lambda(\Omega)}\Big\rVert^p d\lambda\leq C\int_{\Omega} \mathrm{tr}(|Df|^p)d\lambda.
\end{equation*}
For $f\in \mathcal{W}^{1,p}(\Omega,\mathbb{R}^m)$ we shall write
\begin{equation*}
\norm{Df}_{L^p}=\Big(\int_{\Omega} \mathrm{tr}(|Df|^p)d\lambda\Big)^{\frac1p}.
\end{equation*}
%
%
Set
\begin{equation*}
I_{\lambda}^q(h)=\inf\Big\{\norm{H}_{L^q}\mid H\in\Gamma_{\lambda}^q(h)\Big\}.
\end{equation*}
By $\norm{H}_{L^q}$ we mean here the $L^q(\Omega,\mathbb{R}^{m\times n})$ norm with respect to the Schatten $q$-norm on $\mathbb{R}^{m\times n}$, i.e.
\begin{equation*}
\norm{H}_{L^q}=\Big(\int_{\mathbb{R}^n}\mathrm{tr}(|H|^q)d\lambda\Big)^{\frac1q}.
\end{equation*}
We set
\begin{equation*}
J_{\lambda}^p(h)=\sup\Big\{\int_{\mathbb{R}^n}\langle f,h\rangle d\lambda\mid f\in\mathcal{W}^{1,p}(\Omega,\mathbb{R}^m), \norm{Df}_{L^p}\leq 1\Big\}.
\end{equation*}

\begin{theorem}\label{thm:dualityLebesgue}
Let $p,q\in(1,\infty)$ be such that $\frac1p+\frac1q=1$. Let $\Omega\subset\mathbb{R}^n$ be a bounded, open and connected set with Lipschitz boundary. Let
\begin{equation*}
h\in L^q(\Omega,\mathbb{R}^m)
\end{equation*}
be such that
\begin{equation*}
\int_{\Omega}hd\lambda=0.
\end{equation*}
Then
\begin{equation*}
J_{\lambda}^q(h)=I_{\lambda}^q(h).
\end{equation*}
Moreover there exists $f_0\in\mathcal{W}^{1,p}(\Omega,\mathbb{R}^m)$, $\norm{Df_0}_{L^p}\leq1$, such that
\begin{equation*}
J_{\lambda}^q(h)=\int_{\Omega}\langle f_0,h\rangle d\lambda,
\end{equation*}
and there exists $H_0\in\Gamma_{\lambda}^q(h)$ such that
\begin{equation*}
I_{\lambda}^q(h)=\norm{H_0}_{L^q}.
\end{equation*}
Moreover, if $h\neq 0$, then $f_0,H_0$ are optimisers if and only if
\begin{equation}\label{eqn:optifun}
\frac{Df_0^*H_0}{\norm{H_0}_{L^q}}=\frac{H_0^*Df_0}{\norm{H_0}_{L^q}}=\frac{|H_0|^q}{\norm{H_0}_{L^q}^q}=|Df_0|^p
\end{equation}
$\lambda$-almost everywhere.

\end{theorem}
\begin{proof}
Let us observe that if $h\in L^q(\Omega,\mathbb{R}^m)$, then $h\in L^1(\Omega,\mathbb{R}^m)$, as $\Omega$ is bounded. 
Thus the integral $\int_{\Omega}hd\lambda$ is meaningful.
Recall that the dual space of $L^p(\Omega,\mathbb{R}^{m\times n})$ is isometrically isomorphic to $L^q(\Omega,\mathbb{R}^{m\times n})$. Consider a subspace $V$ of $L^p(\Omega,\mathbb{R}^{m\times n})$ given by
\begin{equation*}
V=\Big\{Df\mid f\in\mathcal{W}^{1,p}(\Omega,\mathbb{R}^m)\Big\}
\end{equation*}
and a functional $\Lambda \colon V\to\mathbb{R}$ given by $\Lambda(Df)=\int_{\Omega}\langle f,h\rangle d\lambda$. By the assumptions, it is well-defined and finite. For any function $f\in\mathcal{W}^{1,p}(\Omega,\mathbb{R}^m)$ we have, by the Poincar\'e inequality,
\begin{equation*}
\Lambda(Df)=\int_{\Omega}\Big\langle f-\frac{\int_{\Omega}fd\lambda}{\lambda(\Omega)},h\Big\rangle d\lambda\leq C\norm{h}_{L^q} \norm{Df}_{L^p}.
\end{equation*}
Thus
\begin{equation*}
\norm{\Lambda}\leq C\norm{h}_{L^q}.
\end{equation*}
By the Hahn--Banach theorem we may extend $\Lambda$ to a functional $\Lambda_0$ defined on the space 
 $L^p(\Omega,\mathbb{R}^{m\times n})$ and thus obtain a matrix-valued function $H_0\in L^q(\Omega,\mathbb{R}^{m\times n})$ such that
\begin{equation*}
\int_{\Omega}\langle f, h\rangle d\lambda=\int_{\Omega}\langle Df, H_0\rangle d\lambda
\end{equation*}
for any $f\in\mathcal{W}^{1,p}(\Omega,\mathbb{R}^m)$. Moreover
\begin{equation*}
\norm{H_0}_{L^q}=\norm{\Lambda_0}=\norm{\Lambda}.
\end{equation*}
Since, by Theorem \ref{thm:holder}, $J_{\lambda}^q(h)\leq I_{\lambda}^q(h)$, it follows that
\begin{equation*}
J_{\lambda}^q(h)=\norm{\Lambda}=\norm{H_0}_{L^q}=I_{\lambda}^q(h).
\end{equation*}
We shall now show that there exists $f\in\mathcal{W}^{1,p}(\Omega,\mathbb{R}^m)$, $\norm{Df_0}_{L^p}\leq 1$, with
\begin{equation*}
\int_{\Omega}\langle f_0,h\rangle d\lambda=J_{\lambda}^q(h).
\end{equation*}

Since the Sobolev space $\mathcal{W}^{1,p}(\Omega,\mathbb{R}^m)$ is reflexive, the unit ball in $\mathcal{W}^{1,p}(\Omega,\mathbb{R}^m)$ is weakly compact. The functional
\begin{equation*}
f\mapsto \int_{\Omega}\langle f,h\rangle d\lambda
\end{equation*}
is continuous, as was shown above. By compactness, there exists $f_0$ that maximises its value over the unit ball. This is the function that we were looking for.

For the equality
\begin{equation*}
\int_{\Omega}\langle Df_0,H_0\rangle d\lambda=\norm{H_0}_{L^q}=\Big(\int_{\Omega}\mathrm{tr}(|H_0|^q)\rangle d\lambda\Big)^{\frac1q}
\end{equation*}
to hold true, it is sufficient and necessary that 
\begin{equation*}
\mathrm{tr}(Df_0^*H_0)=\big(\mathrm{tr}(|Df_0|^p)\big)^{\frac1p}\big(\mathrm{tr}(|H_0|^q)\big)^{\frac1q}
\end{equation*}
$\lambda$-almost everywhere and that there exists non-negative constant $c$ such that
\begin{equation*}
\mathrm{tr}(|H_0|^q)=c\mathrm{tr}(|Df_0|^p)
\end{equation*}
$\lambda$-almost everywhere. In view of Theorem \ref{thm:equalitypq} we see that these two conditions may be equivalently stated as
\begin{equation*}
Df_0^*H_0=H_0^*Df_0=d^q|H_0|^q=\frac1{d^p}|Df_0|^p
\end{equation*}
$\lambda$-almost everywhere, for some non-negative constant $d$. If $h\neq 0$, then $H_0\neq 0$ and  also $\norm{Df_0}_{L^p}=1$. Integrating yields
\begin{equation*}
d\norm{H_0}_{L^q}^{\frac1p}=1.
\end{equation*}
It follows that the equivalent condition is that $\lambda$-almost everywhere there is
\begin{equation*}
Df_0^*H_0=H_0^*Df_0=\frac{|H_0|^q}{\norm{H_0}_{L^q}^{\frac{q}{p}}}=\norm{H_0}_{L^q}|Df_0|^p.
\end{equation*}
We infer that (\ref{eqn:optifun}) holds true.
\end{proof}

\begin{remark}\label{rem:possibleflux}
In the definition of divergence of $H\in L^q(\Omega,\mathbb{R}^{m\times n})$ we could have asked for the equality (\ref{eqn:ediv}) to hold true for all $f\in\mathcal{W}^{1,p}_0(\Omega,\mathbb{R}^m)$. 
Then the above theorem holds true as well, with $\mathcal{W}^{1,p}(\Omega,\mathbb{R}^m)$ in lieu of $\mathcal{W}^{1,p}_0(\Omega,\mathbb{R}^m)$. The proof would have to be slightly modified -- we would need to observe that the subspace $\mathcal{W}^{1,p}_0(\Omega,\mathbb{R}^m)\subset \mathcal{W}^{1,p}(\Omega,\mathbb{R}^m)$ is closed, hence reflexive.
\end{remark}
\section{Representation formulae for polar cones}

In this section we shall employ the results of the previous sections. The applications include representation formulae for the dual cone to monotone maps and representation formulae for polar cones to tangent cones to unit ball in space of continuously differentiable maps $\mathcal{C}^1(\Omega,\mathbb{R}^m)$ and in the Sobolev space $\mathcal{W}^{1,p}_0(\Omega,\mathbb{R}^m)$, with $p\in (1,\infty)$ and $\Omega\subset\mathbb{R}^n$ a suitable open set.

\subsection{Dual cone of set of monotone maps}\label{s:polar}

We shall now provide a more general,  alternative proof of the representation formula for the elements of the dual cone of the set of monotone maps. This representation formula is already proven in \cite{Cavalletti4}. Let us remark that the proof is alternative, but in essence both proofs rely on the Hahn--Banach theorem. The new ingredient is employment of matrix H\"older's inequality and Proposition \ref{pro:monotone}.

\begin{definition}
We shall say that a map $u\colon\mathbb{R}^n\to\mathbb{R}^n$ is monotone if
\begin{equation*}
\langle u(x)-u(y),x-y\rangle \geq 0
\end{equation*}
for all $x,y\in\mathbb{R}^n$. 
\end{definition}

\begin{proposition}\label{pro:mono}
A map $u\in\mathcal{C}^1(\mathbb{R}^n,\mathbb{R}^n)$ is monotone if and only if 
\begin{equation}\label{eqn:mono}
\langle Du(x)v,v\rangle\geq 0
\end{equation}
for all $x,v\in\mathbb{R}^n$.
\end{proposition}
\begin{proof}
The proof follows readily.
\end{proof}
\begin{remark}
The condition (\ref{eqn:mono}) means that at any point $x\in\mathbb{R}^n$ the matrix $Du(x)$ is positive semi-definite.
\end{remark}
We shall now study the dual cone of the set of monotone maps, that is the set
\begin{equation*}
\mathcal{P}=\Big\{\mu\in\mathcal{M}(\mathbb{R}^n,\mathbb{R}^n)\mid \int_{\mathbb{R}^n}\langle u,d\mu\rangle\geq  0 \text{ for any monotone }u\in\mathcal{C}^1(\mathbb{R}^n,\mathbb{R}^n) \Big\}.
\end{equation*}
Here $\mathcal{M}(\mathbb{R}^n,\mathbb{R}^n)$ is the space of all $\mathbb{R}^m$-valued measures.

This readily implies also a representation formula for the polar cone of the set of monotone maps, that is the cone
\begin{equation*}
-\mathcal{P}=\Big\{\mu\in\mathcal{M}(\mathbb{R}^n,\mathbb{R}^n)\mid \int_{\mathbb{R}^n}\langle u,d\mu\rangle\leq  0 \text{ for any monotone }u\in\mathcal{C}^1(\mathbb{R}^n,\mathbb{R}^n) \Big\}.
\end{equation*}

For $\mu\in\mathcal{M}(\mathbb{R}^n,\mathbb{R}^n)$, recall the definition of $J(\mu)$; see Section \ref{s:dual}.

\begin{proposition}\label{pro:monotone}
Suppose that $\mu\in \mathcal{M}(\mathbb{R}^n,\mathbb{R}^n)$ is such that $\mu(\mathbb{R}^n)=0$ and has finite first moments. Then $\mu\in \mathcal{P}$ if and only if
\begin{equation}\label{eqn:maxi}
\int_{\mathbb{R}^n}\langle x,d\mu(x)\rangle=J(\mu).
\end{equation}
\end{proposition}
\begin{proof}
Suppose that (\ref{eqn:maxi}) holds true. Let $u\in\mathcal{C}^1(\mathbb{R}^n,\mathbb{R}^n)$ be monotone. Let $\epsilon>0$.  We see that $f\in\mathcal{C}^1(\mathbb{R}^n,\mathbb{R}^m)$ defined by
\begin{equation*}
f(x)=x-\epsilon u(x)
\end{equation*}
satisfies $\norm{f}_{\mathcal{C}^1}\leq \sqrt{1+\epsilon^2\norm{u}_{\mathcal{C}^1}^2}$. Indeed, by monotonicity of $u$, for any $x,v\in\mathbb{R}^n$ we have
\begin{equation*}
\norm{Df(x)v}^2=\norm{v}^2+\epsilon^2 \norm{Du(x)v}-2\epsilon\langle v,Du(x)v\rangle\leq \norm{v}^2+\epsilon^2\norm{Du}^2\norm{v}^2.
\end{equation*} 
Let $\frac{1}{c}=\sqrt{1+\epsilon^2\norm{u}_{\mathcal{C}^1}^2}$. Then, as $cf$ is $1$-Lipschitz, by (\ref{eqn:maxi}),
\begin{equation*}
\int_{\mathbb{R}^n}\langle x-cf(x),d\mu(x)\rangle \geq 0
\end{equation*}
and therefore
\begin{equation*}
\int_{\mathbb{R}^n}\langle u,d\mu\rangle \geq \frac{\epsilon\norm{u}_{\mathcal{C}^1}^2} {1+\sqrt{1+\epsilon^2\norm{u}_{\mathcal{C}^1}^2}}\int_{\mathbb{R}^n}\langle x,d\mu(x)\rangle
\end{equation*}
for any $\epsilon>0$. Letting $\epsilon$ tend to zero, we obtain $\mu\in\mathcal{P}$.

Assume now that $\mu\in \mathcal{P}$. We shall show that (\ref{eqn:maxi}) holds true. Take any map $h\in\mathcal{C}^1(\mathbb{R}^n,\mathbb{R}^n)$ with $\norm{h}_{\mathcal{C}^1}\leq 1$. Then for any $x\in\mathbb{R}^n$ and any $v\in\mathbb{R}^n$
\begin{equation*}
\langle v-Dh(x)v,v\rangle \geq 0.
\end{equation*}
Thus, by Proposition \ref{pro:mono}, $\mathrm{id}-h$ is monotone. As $-\mu\in \mathcal{P}$, we have
\begin{equation*}
\int_{\mathbb{R}^n}\langle x,d\mu(x)\rangle \geq \int_{\mathbb{R}^n}\langle h,d\mu\rangle.
\end{equation*}
\end{proof}

Let us denote by $\mathcal{S}^{n\times n}_+$  the set of all $n\times n$ positive semi-definite symmetric matrices.
\begin{theorem}\label{thm:positive}
For any $\mu\in \mathcal{P}$ with finite first moments, there exists a matrix-valued measure $M\in\mathcal{M}(\mathbb{R}^n,\mathcal{S}^{n\times n}_+)$ such that 
\begin{equation*}
\mu=-divM
\end{equation*}
and
\begin{equation*}
\int_{\mathbb{R}^n}\mathrm{tr}(dM)=\int_{\mathbb{R}^n}\langle x,d\mu(x)\rangle.
\end{equation*}
Conversely, if a divergence of an $\mathcal{S}^{n\times n}_+$-valued measure is a measure, then it belongs to $\mathcal{P}$.
\end{theorem}
\begin{proof}
Suppose that $\mu\in \mathcal{P}$ has finite first moments. Then we know that 
\begin{equation*}
\int_{\mathbb{R}^n}\langle x,d\mu(x)\rangle=J(\mu).
\end{equation*}
Let $M_0$ be such that $\norm{M_0}_{\mathcal{M}}$ is minimal on $\Xi(\mu)$, i.e. on the set of measures $M$ such that
\begin{equation*}
-divM=\mu.
\end{equation*}
Then by Theorem \ref{thm:duality3} we know that
\begin{equation*}
\Big\langle \mathrm{Id},\frac{dM_0}{d\norm{M_0}_1}\Big\rangle =1
\end{equation*}
$\norm{M_0}_1$-almost everywhere. By Example \ref{exa:identity} it follows that this holds if and only if $\frac{dM_0}{d\norm{M_0}_1}$ is symmetric and positive semi-definite. Then,
\begin{equation*}
\int_{\mathbb{R}^n}\langle x,d\mu(x)\rangle=\norm{M_0}_{\mathcal{M}}= \int_{\mathbb{R}^n}\mathrm{tr}(dM_0).
\end{equation*}
This proves the first part of the theorem. 
For the second, let $M\in\mathcal{M}(\mathbb{R}^n,\mathcal{S}^{n\times n}_+)$ be such that $\mu=-div M$. Then, as $M$ is positive semi-definite,
\begin{equation*}
\Big\langle \mathrm{Id},\frac{dM}{d\norm{M}_1}\Big\rangle =\Big\lVert \frac{dM}{d\norm{M}_1}\Big\rVert_1=1.
\end{equation*}
It follows that 
\begin{equation*}
J(\mu)\geq \int_{\mathbb{R}^n}\langle x,d\mu(x)\rangle=\norm{M}_{\mathcal{M}}\geq J(\mu).
\end{equation*}
and by Proposition \ref{pro:monotone} we conclude the proof.
\end{proof}

\begin{remark}
The advantage of the proof above is that the representation is a direct consequence of the duality formula and Example \ref{exa:identity}. Thus the proof may be adapted to computation of the representations of other dual cones.
\end{remark}


\subsection{Polar cones to tangent cones of $\mathcal{C}^1(\mathbb{R}^n,\mathbb{R}^m)$}\label{s:tangent}

Below we shall apply Theorem \ref{thm:duality3} to compute the representations of polar cones to certain tangent cones to the unit ball of $\mathcal{C}^1(\mathbb{R}^n,\mathbb{R}^m)$.

\begin{definition}
For $f\in\mathcal{C}^1(\mathbb{R}^n,\mathbb{R}^m)$, $\norm{f}_{\mathcal{C}^1}=1$ define
\begin{equation*}
\mathcal{P}_f=\Big\{\mu\in\mathcal{M}(\mathbb{R}^n,\mathbb{R}^m)\mid \mu(\mathbb{R}^n)=0, \int_{\mathbb{R}^n}\norm{x}d\norm{\mu}(x)<\infty, J(\mu)=\int_{\mathbb{R}^n}\langle f,d\mu\rangle \Big\}.
\end{equation*}
\end{definition}

Let us note that the set $\mathcal{P}_f$ is the polar cone to the tangent cone of the unit ball at $f\in\mathcal{C}^1(\mathbb{R}^n,\mathbb{R}^m)$.

\begin{proposition}\label{pro:gene}
Suppose that $f\in\mathcal{C}^1(\mathbb{R}^n,\mathbb{R}^m)$, $\norm{f}_{\mathcal{C}^1}=1$. Measure $\mu$ with finite first moments belongs to $\mathcal{P}_f$ if and only if for any map $h\in \mathcal{C}^1(\mathbb{R}^n,\mathbb{R}^m)$ such that for $\epsilon>0$ the Lipschitz constants $\lambda_{\epsilon}$ of $f-\epsilon h$ satisfy
\begin{equation}\label{eqn:polar}
\liminf_{\epsilon\to 0^+}\frac{1-\lambda_{\epsilon}}{\epsilon}\geq 0
\end{equation}
there is $\int_{\mathbb{R}^n}\langle h,d\mu\rangle\geq 0$.
\end{proposition}
\begin{proof}
Let $\mu\in\mathcal{P}_f$. Suppose that for $h\in \mathcal{C}^1(\mathbb{R}^n,\mathbb{R}^m)$ condition (\ref{eqn:polar}) holds true. 
For such $h$, consider 
\begin{equation*}
g=f-\epsilon h.
\end{equation*}
Then $\frac{1}{\lambda_{\epsilon}}g$ is $1$-Lipschitz and hence
\begin{equation*}
\int_{\mathbb{R}^n}\langle f,d\mu\rangle\geq \frac1{\lambda_{\epsilon}}\int_{\mathbb{R}^n}\langle g,d\mu\rangle.
\end{equation*}
It follows that 
\begin{equation*}
\int_{\mathbb{R}^n}\langle h,d\mu\rangle\geq \frac{1-\lambda_{\epsilon}}{\epsilon}\int_{\mathbb{R}^n}\langle f,d\mu\rangle.
\end{equation*}
Note that $\int_{\mathbb{R}^n}\langle f,d\mu\rangle\geq 0$ -- otherwise $-f$ would yield greater value of the integral. Thus the assertion follows by taking limit. 

Conversely, suppose that for any $h\in\mathcal{C}^1(\mathbb{R}^n,\mathbb{R}^m)$ such that (\ref{eqn:polar}) holds true there is $\int_{\mathbb{R}^n}\langle h,d\mu\rangle\geq 0$. Choose any $g\in\mathcal{C}^1(\mathbb{R}^n,\mathbb{R}^m)$ with $\norm{g}_{\mathcal{C}^1}\leq 1$. Let $h=f-g$. Then the corresponding Lipschitz constant of $f-\epsilon h=(1-\epsilon)f+\epsilon g$ is at most one. 
Hence, the condition (\ref{eqn:polar}) is satisfied for $h$ and the claim follows.
\end{proof}

The proposition above tells us that $\mathcal{P}_f$ is the dual cone to the convex cone $\mathcal{H}_f$ of functions $h\in\mathcal{C}^1(\mathbb{R}^n,\mathbb{R}^n)$ that satisfy (\ref{eqn:polar}). 
Let us show that  $\mathcal{H}_f$ is indeed a convex cone.
Clearly, if $h\in\mathcal{H}_f$, then also $\lambda h\in\mathcal{H}_f$ for any non-negative $\lambda$. If $h_1,h_2\in\mathcal{H}_f$, then $\frac{h_1+h_2}2\in\mathcal{H}_f$, as Lipschitz constant of a convex combination of functions is at most the convex combination of the corresponding Lipschitz constants.

Let us note that, when $f$ is the identity map, the condition (\ref{eqn:polar}) simplifies to demand that $h$ is monotone.

\begin{theorem}\label{thm:cones}
Suppose that $f\in\mathcal{C}^1(\mathbb{R}^n,\mathbb{R}^m)$, $\norm{f}_{\mathcal{C}^1}=1$. For any $\mu\in \mathcal{P}_f$ there exists a measure $M\in\mathcal{M}(\mathbb{R}^n,\mathbb{R}^{m\times n})$ with
\begin{equation*}
-divM=\mu,
\end{equation*}
and such that $\norm{M}_1$-almost everywhere
\begin{equation}\label{eqn:isometry}
 DfDf^*=\mathrm{Id}\text{ on }\mathrm{im}\Big(\frac{dM}{d\norm{M}_1}\Big),
\end{equation}
and
\begin{equation}\label{eqn:symm}
Df^*\frac{dM}{d\norm{M}_1}\text{ is symmetric and positive semidefinite.}
\end{equation}
Moreover
\begin{equation*}
\norm{M}_{\mathcal{M}}=\int_{\mathbb{R}^n}\langle f,d\mu\rangle.
\end{equation*}
Conversely, if $M\in\mathcal{M}(\mathbb{R}^n,\mathbb{R}^{m\times n})$ is such that $-divM$ is a finite vector-valued measure such that (\ref{eqn:isometry}) and (\ref{eqn:symm}) are satisfied then $-divM\in \mathcal{P}_f$.
\end{theorem}
\begin{proof}
Follows from Theorem \ref{thm:equality} and from Theorem \ref{thm:duality3}, c.f. proof of Theorem \ref{thm:positive}.
\end{proof}

\subsection{Polar cones to tangent cones of $\mathcal{W}^{1,p}(\Omega,\mathbb{R}^m)$}\label{s:tangentabs}

The method developed above may be as well applied in the context of absolutely continuous vector measures with use of results of Section \ref{s:abs}. 

In what follows we shall understand that $-divH=h$ if (\ref{eqn:ediv}) holds true for all $f\in \mathcal{W}^{1,p}(\Omega,\mathbb{R}^m)$.

\begin{definition}
Let $p,q\in (1,\infty)$ be such that $\frac1p+\frac1q=1$ and let $\Omega\subset\mathbb{R}^n$ be an open set. For $f\in \mathcal{W}^{1,p}(\Omega,\mathbb{R}^m)$, $\norm{Df}_{L^p}=1$ define
\begin{equation*}
\mathcal{R}_f=\Big\{h\in L^q(\Omega,\mathbb{R}^m)\mid J^q_{\lambda}(h)=\int_{\Omega}\langle f,h\rangle d\lambda\Big\}.
\end{equation*}
\end{definition}

In the above definition $\mathcal{R}_f$ is the polar cone to the tangent cone of the unit ball of $\mathcal{W}^{1,p}(\Omega,\mathbb{R}^m)$.

\begin{proposition}\label{pro:abso}
Let $p,q\in (1,\infty)$ be such that $\frac1p+\frac1q=1$ and let let $\Omega\subset\mathbb{R}^n$ be an open set. Let $f\in\mathcal{W}^{1,p}(\Omega,\mathbb{R}^m)$ be such that $\norm{Df}_{L^p}=1$.
Then $h\in L^q(\Omega,\mathbb{R}^m)$ belongs to $\mathcal{R}_f$ if and only if for any map $g\in\mathcal{W}^{1,p}(\Omega,\mathbb{R}^m)$ such that 
\begin{equation*}
\liminf_{\epsilon\to 0^+}\frac{1-\norm{Df-\epsilon Dg}_{L^p}}{\epsilon}\geq 0
\end{equation*}
there is $\int_{\Omega}\langle h,g\rangle d\lambda\geq 0$.
\end{proposition}
\begin{proof}
The proof follows analogous lines to the lines of the proof of Proposition \ref{pro:gene}.
\end{proof}

\begin{theorem}\label{thm:abso}
Let $p,q\in (1,\infty)$ be such that $\frac1p+\frac1q=1$ and let $\Omega\subset\mathbb{R}^n$ be an open, bounded and connected set with Lipschitz boundary. Let $f\in\mathcal{W}^{1,p}(\Omega,\mathbb{R}^m)$ be such that $\norm{Df}_{L^p}=1$. For any $h\in \mathcal{R}_f$, there exists $H\in L^q(\Omega,\mathbb{R}^{m\times n})$ such that
\begin{equation*}
-div H=h
\end{equation*}
and such that $\lambda$-almost everywhere
\begin{equation}\label{eqn:divv}
\frac{Df^*H}{\norm{H}_{L^q}}=\frac{H^*Df}{\norm{H}_{L^q}}=|Df|^p=\frac{|H|^q}{\norm{H}_{L^q}^q}.
\end{equation}
Conversely if $H\in L^q(\Omega,\mathbb{R}^{m\times n})$ is such that $-divH$ is a function in $L^q(\Omega,\mathbb{R}^m)$, that satisfies (\ref{eqn:divv}), then $-divH\in\mathcal{R}_f$.
\end{theorem}
\begin{proof}
Again, the proof relies on Theorem \ref{thm:dualityLebesgue} and the strategy does not differ from the strategies in Theorem \ref{thm:positive} and in Theorem  \ref{thm:cones}.
\end{proof}

Let us apply the above results to a particular case when $p=2$. We shall below compute the dual cone to the cone of maps $g\in\mathcal{W}^{1,2}(\Omega,\mathbb{R}^m)$ such that $\int_{\Omega}\langle Dg,Df\rangle d\lambda\geq 0$.

\begin{corollary}\label{col:peq2}
Let $\Omega\subset\mathbb{R}^n$ be an open, bounded and connected set with Lipschitz boundary. Let $f\in\mathcal{W}^{1,2}(\Omega,\mathbb{R}^m)$, $\norm{Df}_{L^2}=1$. For $h\in L^2(\Omega,\mathbb{R}^m)$, the following conditions are equivalent:
\begin{enumerate}[label=\roman*)]
\item there exists $H\in L^2(\Omega,\mathbb{R}^{m\times n})$ such that $-div H=h$ and $\lambda$-almost everywhere there is
\begin{equation}\label{eqn:cooo}
\frac{Df^*H}{\norm{H}_{L^2}}=\frac{H^*Df}{\norm{H}_{L^2}}=|Df|^2=\frac{|H|^2}{\norm{H}_{L^2}^2},
\end{equation}
\item for any $g\in\mathcal{W}^{1,2}(\Omega,\mathbb{R}^m)$ such that 
\begin{equation*}
\int_{\Omega}\langle Dg, Df\rangle d\lambda\geq 0
\end{equation*}
there is $\int_{\Omega}\langle g,h\rangle d\lambda\geq 0$.
\end{enumerate}
\end{corollary}
\begin{proof}
Observe that if $p=2$, then Proposition \ref{pro:abso} yields that $\mathcal{R}_f$ consists of exactly these $h\in L^2(\Omega,\mathbb{R}^m)$ such that if $g\in\mathcal{W}^{1,2}(\Omega,\mathbb{R}^m)$ satisfies
\begin{equation*}
\int_{\Omega}\langle Df,Dg\rangle d\lambda\geq 0
\end{equation*}
then $\int_{\Omega}\langle h,g\rangle d\lambda\geq 0$.
Indeed, we have for $\epsilon>0$,
\begin{equation*}
\frac{1-\norm{Df-\epsilon Dg}_{L^2}}{\epsilon}=\frac{1-\norm{Df-\epsilon Dg}_{L^2}^2}{\epsilon(1+\norm{Df-\epsilon Dg}_{L^2})}= \frac{-\epsilon \norm{Dg}_{L^2}+2 \int_{\Omega}\langle Df,Dg\rangle d\lambda}{1+\norm{Df-\epsilon Dg}_{L^2}} 
\end{equation*}
and this quantity converges to $\int_{\Omega} \langle Df,Dg\rangle d\lambda$ as $\epsilon$ tends to zero.
Now, Theorem \ref{thm:abso} tells us that there exists $H\in L^2(\Omega,\mathbb{R}^{m\times n})$ such that $-div H=h$ and such that (\ref{eqn:cooo}) holds true.
\end{proof}

Note that the limit that is under investigation in Proposition \ref{pro:abso} is actually computable for all values of $p\in (1,\infty)$. It is the lower Dini derivative of the norm of $f$ taken in direction of $g$. 

\begin{theorem}\label{thm:polarsob}
Let $\Omega\subset\mathbb{R}^n$ be an open, bounded and connected set with Lipschitz boundary. Let $p\in (1,\infty)$, $\frac1p+\frac1q=1$, $f\in\mathcal{W}^{1,p}(\Omega,\mathbb{R}^m)$,  $\norm{Df}_{L^p}=1$. For a map $h\in L^q(\Omega,\mathbb{R}^m)$, the following conditions are equivalent:
\begin{enumerate}[label=\roman*)]
\item there exists $H\in L^q(\Omega,\mathbb{R}^{m\times n})$ such that $-div H=h$ and $\lambda$-almost everywhere
\begin{equation}\label{eqn:divve}
\frac{Df^*H}{\norm{H}_{L^q}}=\frac{H^*Df}{\norm{H}_{L^q}}=|Df|^p=\frac{|H|^q}{\norm{H}_{L^q}^q},
\end{equation}
\item\label{i:two} for any $g\in\mathcal{W}^{1,p}(\Omega,\mathbb{R}^m)$ such that 
\begin{equation*}
\int_{\Omega}\mathrm{tr}\Big( (DfDf^*)^{\frac{p}2-1}\big( DfDg^*+DgDf^*\big)\Big) d\lambda\geq 0
\end{equation*}
there is $\int_{\Omega}\langle g,h\rangle d\lambda\geq 0$.
\end{enumerate}
\end{theorem}
\begin{proof}
Observe that Proposition \ref{pro:abso} yields that $\mathcal{R}_f$ consists of exactly these maps $h\in L^q(\Omega,\mathbb{R}^m)$ such that \ref{i:two} is satisfied.
Indeed, the limit
\begin{equation*}
\lim_{\epsilon\to 0}\frac{1-\norm{Df-\epsilon Dg}_{L^p}}{\epsilon}
\end{equation*}
exists and it is equal to, by the chain rule,
\begin{equation}\label{eqn:outcome}
\frac12 \int_{\Omega}\mathrm{tr}\Big( (DfDf^*)^{\frac{p}2-1}\big( DfDg^*+DgDf^*\big)\Big) d\lambda.
\end{equation}
For this, we need to compute the following limit
\begin{equation}\label{eqn:limit}
\lim_{\epsilon\to 0}\int_{\Omega}\mathrm{tr}\Big(\frac1{\epsilon}\Big((Df+\epsilon Dg)(Df+\epsilon Dg)^*\Big)^{\frac{p}2}-(DfDf^*)^{\frac{p}2}\Big) d\lambda.
\end{equation}
The integrand is a differential quotient, hence it is equal to the derivative at some point, that is 
\begin{equation*}
\frac{p}2\mathrm{tr}\bigg(\Big(( Df+\epsilon_0Dg)(Df+\epsilon_0Dg)^*\Big)^{\frac{p}2-1}\Big((Df+\epsilon_0Dg)Dg^*+Dg(Df+2\epsilon_0Dg)^*\Big)\bigg)
\end{equation*}
for some $\epsilon_0\in [0,\epsilon]$, which may depend on a point in $\Omega$. Let us denote by $A$ the sum $Df+\epsilon_0Dg$. Matrix H\"older's inequality and the cyclic property of trace tell us that the absolute value of this expression is bounded from above by
\begin{equation}\label{eqn:expression}
\frac{p}2\abs{Dg}_p\Big( \big|(AA^*)^{\frac{p}2-1}A\big|_q+\big|A^*(AA^*)^{\frac{p}2-1}\big|_q\Big)
\end{equation}
To compute the above observe that
\begin{equation*}
 \big|(AA^*)^{\frac{p}2-1}A\big|_q^q=\mathrm{tr}\Big(\big((AA^*)^{\frac{p}2-1}AA^*(AA^*)^{\frac{p}2-1}\big)^{\frac{q}2}\Big)=\mathrm{tr}\Big( (AA^*)^{\frac{(p-1)q}2}\Big)=\mathrm{tr} (AA^*)^{\frac{p}2}
\end{equation*}
since $(p-1)q=p$. Proceeding analogously with the second summand we get that (\ref{eqn:expression}) is equal to
\begin{equation}\label{eqn:integrand}
p\abs{Dg}_p \abs{Df+\epsilon_0Dg}_p^\frac{p}{q}.
\end{equation}
 By triangle inequality and assuming that $\epsilon\leq 1$, we see that (\ref{eqn:integrand}) may be bounded by
\begin{equation*}
p\abs{Dg}_p \big(\abs{Df}_p+\epsilon\abs{Dg}_p\big)^\frac{p}{q}\leq 2^{\frac{p}{q}}p\abs{Dg}_p\big(\abs{Df}_p^\frac{p}{q}+\abs{Dg}_p^\frac{p}{q}\big).
\end{equation*}
H\"older's inequality for integrals tells us that the function on the right hand-side of the above inequality is integrable. 
Therefore use of dominated convergence theorem is justified. It follows that the limit (\ref{eqn:limit}) is indeed equal to (\ref{eqn:outcome}), up to a factor $\frac{p}{2}$. Clearly, we have also proven that the integrand in (\ref{eqn:outcome}) is integrable.

Now, Theorem \ref{thm:abso} tells us that there exists $H\in L^2(\Omega,\mathbb{R}^{m\times n})$ with negative divergence equal to $h$ and such that (\ref{eqn:divve}) holds true.
The proof is complete.
\end{proof}

\begin{remark}
In view of Remark \ref{rem:possibleflux} all the results in the current section hold true also for functions in $\mathcal{W}^{1,p}_0(\Omega,\mathbb{R}^m)$, in lieu of $\mathcal{W}^{1,p}(\Omega,\mathbb{R}^m)$.
\end{remark}

\section*{Acknowledgements}
The author would like to express his thanks to the referees, whose remarks allowed for an improvement of the manuscript.
\bibliographystyle{siamplain}
\bibliography{refs}
\end{document}